\documentclass[11pt]{amsart}
\usepackage[T1]{fontenc}
\usepackage{amssymb,amsmath, amsthm, amsfonts, tikz-cd}

\usepackage{graphicx}
\usepackage{listings}
\usepackage[margin=.75in]{geometry}
\usepackage{lstautogobble}
\usepackage{enumerate}
\usepackage[shortlabels]{enumitem}
\usepackage{thmtools}
\usepackage{thm-restate}
\usepackage{amsthm}
\usepackage{verbatim}

\usepackage{mathtools}
\usepackage{physics}
\usepackage{color}
\usepackage{hyperref}
\usepackage[capitalise]{cleveref}
\usepackage[final]{showlabels}
\usepackage[bottom]{footmisc}
\crefformat{equation}{(#2#1#3)}

\usepackage{float}
\restylefloat{table}

\usepackage{mathrsfs}
\setlist{  
  listparindent=\parindent,
  parsep=0pt,
}

\theoremstyle{plain}
\newtheorem{thm}{Theorem}[section]
\newtheorem{prop}[thm]{Proposition}
\newtheorem{lemma}[thm]{Lemma}
\newtheorem{cor}[thm]{Corollary}

\theoremstyle{definition}
\newtheorem{mydef}[thm]{Definition}

\newtheorem{remark}[thm]{Remark}

\numberwithin{equation}{section} 

\DeclarePairedDelimiter\ipp{\langle}{\rangle}

\DeclarePairedDelimiter{\paren}{\lparen}{\rparen}

\DeclarePairedDelimiter{\jp}{\langle}{\rangle}

\DeclareMathOperator{\supp}{supp}

\newcommand{\M}{{\mathcal{M}}}

\newcommand{\p}{{\partial}}

\renewcommand{\d}{\delta}

\newcommand{\R}{{\mathbb{R}}}

\newcommand{\N}{{\mathbb{N}}}

\newcommand{\Z}{{\mathbb{Z}}}

\renewcommand{\H}{{\mathcal{H}}}
\renewcommand{\P}{{\mathcal{P}}}
\newcommand{\T}{{\mathbb{T}}}
\newcommand{\E}{\mathbf{E}}
\newcommand{\g}{{\mathfrak{g}}}

\newcommand{\f}{\mathfrak{f}}
\newcommand{\Fr}{{\mathfrak{F}}}
\newcommand{\Hr}{\mathfrak{H}}

\newcommand{\Sc}{{\mathcal{S}}} 
\newcommand{\Dc}{\mathcal{D}} 

\renewcommand{\M}{{\mathcal{M}}}

\newcommand{\U}{\mathfrak{U}}

\newcommand{\tl}{\tilde}

\newcommand{\D}{\Delta}

\newcommand{\ph}{\phantom{=}}
\newcommand{\nn}{\nonumber}

\newcommand{\ol}{\overline}

\newcommand{\ux}{\underline{x}}
\newcommand{\uz}{\underline{z}}

\newcommand{\ur}{\underline{r}}

\newcommand{\ep}{\epsilon}
\newcommand{\vep}{\varepsilon}
\newcommand{\al}{\alpha}
\newcommand{\be}{\beta}

\newcommand{\Uu}{\mathfrak{U}}

\newcommand{\na}{\nabla}

\newcommand{\Te}{\mathrm{Term}}

\newcommand{\wh}{\widehat}
\newcommand{\ueta}{\underline{\eta}}

\renewcommand{\ueta}{\underline{\eta}}

\newcommand{\Dm}{|\nabla|}

\let\oldtocsection=\tocsection
 
\let\oldtocsubsection=\tocsubsection
 
\let\oldtocsubsubsection=\tocsubsubsection
 
\renewcommand{\tocsection}[2]{\hspace{0em}\oldtocsection{#1}{#2}}
\renewcommand{\tocsubsection}[2]{\hspace{1em}\oldtocsubsection{#1}{#2}}
\renewcommand{\tocsubsubsection}[2]{\hspace{2em}\oldtocsubsubsection{#1}{#2}}

\begin{document}

\title[Rigorous Derivation of Incompressible Euler Equation from Newton's Second Law]{On the Rigorous Derivation of the Incompressible Euler Equation from Newton's Second Law}

\author[M. Rosenzweig]{Matthew Rosenzweig}
\address{  
Department of Mathematics\\ 
Massachusetts Institute of Technology\\
Headquarters Office\\
Simons Building (Building 2), Room 106\\
77 Massachusetts Ave\\
Cambridge, MA 02139-4307}
\email{mrosenzw@mit.edu}

\begin{abstract}
A longstanding problem in mathematical physics is the rigorous derivation of the incompressible Euler equation from Newtonian mechanics. Recently, Han-Kwan and Iacobelli \cite{HkI2020} showed that in the monokinetic regime, one can directly obtain the Euler equation from a system of $N$ particles interacting in $\T^d$, $d\geq 2$, via Newton's second law through a \emph{supercritical mean-field limit}. Namely, the coupling constant $\lambda$ in front of the pair potential, which is Coulombic, scales like $N^{-\theta}$ for some $\theta \in (0,1)$, in contrast to the usual mean-field scaling $\lambda\sim N^{-1}$. Assuming $\theta\in (1-\frac{2}{d(d+1)},1)$, they showed that the empirical measure of the system is effectively described by the solution to the Euler equation as $N\rightarrow\infty$. Han-Kwan and Iacobelli asked if their range for $\theta$ was optimal. We answer this question in the negative by showing the validity of the incompressible Euler equation in the limit $N\rightarrow\infty$ for $\theta \in (1-\frac{2}{d},1)$. For reasons of scaling, this range appears optimal in all dimensions. Our proof is based on Serfaty's modulated-energy method, but compared to that of Han-Kwan and Iacobelli, crucially uses an improved ``renormalized commutator'' estimate to obtain the larger range for $\theta$.
\end{abstract}

\maketitle

\section{Introduction}
\label{sec:intro}
\subsection{Background}
\label{ssec:introback}
A source of much research in mathematical physics is the problem of rigorously deriving the \emph{incompressible Euler equation} in dimensions $d\geq 2$
\begin{equation}
\label{eq:Eul}
\begin{cases}
\p_t u + (u\cdot\nabla)u = -\nabla p\\
\nabla\cdot{u} =0 \\
u|_{t=0} = u_0,
\end{cases}
\qquad (t,x)\in \R\times\T^d,
\end{equation}
which describes the evolution of the velocity field $u$ of an ideal fluid ($p$ is the scalar pressure), from \emph{Newton's laws of mechanics} for the motion of $N$ indistinguishable particles with binary interactions
\begin{equation}
\label{eq:New}
\begin{cases}
\dot{x}_i = v_i \\
\dot{v}_i = -\lambda\sum_{{1\leq j\leq N}\atop{j\neq i}} \nabla\g(x_i-x_j).
\end{cases}
\end{equation}
Here, $(x_i,v_i)$ denote the position and momentum, respectively, of the $i$-th particle, in the phase space $\T^d\times\R^d$; $\g$ is the interaction pair potential; and $\lambda$ is a scaling parameter which encodes physical information about the system and about which we shall say more momentarily. We assume $\g$ to be the Coulomb potential (i.e. the Green's function of $-\D$ normalized to have zero mean). Formally, one can derive Euler's equation by considering the fluid as a continuum and applying Newton's second law to infinitesimal fluid volume elements. See, for instance, \cite{Tao2018}. However, turning such heuristic considerations into a mathematically rigorous proof is challenging.

Viewing Euler's equation as a macroscopic description and Newton's law as a microscopic description, one strategy is to go from Newton to Euler by first passing to a mesoscopic description, namely Boltzmann's equation for the evolution of the distribution function in particle phase space. By considering a suitable hydrodynamic scaling regime, one can then derive solutions to the incompressible Euler equation from the Boltzmann equation. Much research has been done on this topic, and we refer the interested reader to Saint-Raymond's monograph \cite{SR2009} for more details.

Recently, Han-Kwan and Iacobelli \cite{HkI2020} rigorously derived Euler's equation \eqref{eq:Eul} through a \emph{supercritical mean-field limit} of the Newtonian $N$-body problem \eqref{eq:New} in the \emph{monokinetic} regime. More precisely, they start from the many-body problem \eqref{eq:New} with coupling constant $\lambda = N^{-\theta}$, for $\theta>0$. The term monokinetic refers to the assumption that the velocities $v_i \approx u(x_i)$, for the same vector field $u$. Note that since we are in the repulsive setting, the particles remain separated and have bounded velocities, assuming the particles are initially separated. Consequently, the dynamics of \eqref{eq:New} are globally well-posed. Under the restriction
\begin{equation}
\label{eq:th_res}
1-\frac{2}{d(d+1)} < \theta < 1
\end{equation}
and assuming that the empirical measure of the system \eqref{eq:New},
\begin{equation}
\label{eq:EM}
\frac{1}{N}\sum_{i=1}^N \d_{(x_i(t),v_i(t))}(x,v)
\end{equation}
converges at time $t=0$ in a suitable topology to the measure 
\begin{equation}
\label{eq:limEM}
\d_0(u(t,x)-v)dxdv,
\end{equation}
where $u$ is a classical solution to equation \eqref{eq:Eul}, Han-Kwan and Iacobelli show using the modulated-energy method of \cite{Serfaty2020} that the empirical measure converges in the weak-* topology to the measure \eqref{eq:limEM} on the lifespan of $u$.

The modifier ``supercritical,'' coined by Han-Kwan and Iacobelli, stems from the following trichotomy.
\begin{itemize}[leftmargin=*]
\item
In the \emph{subcritical mean-field} regime $\theta>1$, the force term formally vanishes as $N\rightarrow\infty$, assuming that each term in the sum is $O(1)$, and therefore interactions become negligible for a very large number of particles.
\item
In the \emph{mean-field} regime $\theta=1$, the force is $O(1)$ as $N\rightarrow\infty$, and one expects the empirical measure of the system \eqref{eq:New} to converge to a solution of the \emph{pressureless Euler-Poisson equation} as $N\rightarrow\infty$. Serfaty, in collaboration with Duerinckx, recently proved this convergence via the aforementioned modulated-energy method in the breakthrough paper \cite[Appendix]{Serfaty2020}.\footnote{This result of Serfaty and Duerinckx is a special case of a result covering the full range of Riesz potential interactions. The limiting equation for super-Coulombic Riesz interactions is the so-called \emph{pressureless Euler-Riesz system}, the well-posedness of which has been studied by Choi and Jeong \cite{CJ2020}.} We also mention that the mean-field limit outside of the monokinetic regime is of great interest, as the limiting equation is \emph{Vlasov-Poisson}, a complete derivation of which remains elusive, despite much activity in recent years (e.g. see \cite{Spohn1991, HJ2007, Hauray2014, HJ2015, Golse2016ln, BP2016, Lazarovici2016, LP2017} and references therein).  
\item
In the \emph{supercritical mean-field} regime $0<\theta<1$, the force term may, in principle, diverge as $N\rightarrow\infty$, leaving open the possibility of more singular behavior than in the mean-field regime.
\end{itemize}

As they observe, one can equivalently interpret the result of Han-Kwan and Iacobelli \cite{HkI2020} as convergence in the combined mean-field and \emph{quasineutral} limits. More precisely, one can introduce a parameter $\vep>0$ by setting $\lambda=\vep^2 N = N^{\theta}$, so that the system \eqref{eq:New} becomes
\begin{equation}
\label{eq:NewQN}
\begin{cases}
\dot{x}_i = v_i \\
\dot{v}_i = -\frac{1}{\vep^2 N}\sum_{{1\leq j\leq N}\atop {j\neq i}} \nabla\g(x_i-x_j),
\end{cases}
\end{equation}
and the parameter $\vep$ now has the meaning of the \emph{Debye length}, which is the typical length scale of the interactions. The limit $\vep\rightarrow 0$ is called the quasineutral limit. The mean-field limit $N\rightarrow \infty$ of the system \eqref{eq:NewQN} formally leads to the Vlasov-Poisson equation, as discussed above. Several authors \cite{BG1994, Grenier1995, Grenier1996, Brenier2000, Masmoudi2001, HkH2015, HkR2016, HkI2017jde, HkI2017cms, GpI2020, GpI2020dev} have considered the quasineutral limit of Vlasov-Poisson with varying assumptions on the initial datum. But in the monokinetic regime of present interest, Brenier \cite{Brenier2000} has rigorously shown that the quasineutral limit leads to the incompressible Euler equation \eqref{eq:Eul}. We also mention that the combined mean-field and quasineutral regime has been studied in \cite{GpI2018, GpI2020} without the monokinetic assumption, where the expected limiting equation is the so-called \emph{kinetic Euler equation}. These results, though, are only slightly supercritical, in the sense that the coupling constant $\lambda\sim (\ln\ln N)/N$ as $N\rightarrow\infty$. 

\subsection{Statement of main results}
\label{ssec:introMR}
A question left open by Han-Kwan and Iacobelli in \cite{HkI2020} is whether the restriction \eqref{eq:th_res} is optimal. In the present article, we answer this question in the negative in all dimensions $d\geq 2$ by improving the range of $\theta$ to
\begin{equation}
\label{eq:thrng}
1-\frac{2}{d} < \theta < 1.
\end{equation}
As explained below, scaling considerations suggest that this range is optimal. To state our main results, we introduce the \emph{modulated energy} from \cite{HkI2020} (cf. the expression in \cite[pg. 40]{Serfaty2020})
\begin{equation}
\label{eq:MEdef}
\Hr_{N,\vep}(\uz_N,u) \coloneqq \frac{1}{2N}\sum_{i=1}^N|u(x_i) - v_i|^2 + \frac{1}{2\vep^2}\Fr_N(\ux_N,1+\vep^2 \Uu),
\end{equation}
where
\begin{equation}
\label{eq:MEpdef}
\Fr_N(\ux_N,1+\vep^2\Uu) \coloneqq \int_{(\T^d)^2\setminus\D_2} \g(x-y)d(\frac{1}{N}\sum_{i=1}^N\d_{x_i} - 1-\vep^2\Uu)^{\otimes 2}(x,y).
\end{equation}
Above, $u$ is a solution to equation \eqref{eq:Eul}, $\U\coloneqq \p_{\al}u^{\be} \p_{\be}u^\al=-\D p$, and $\uz_N \coloneqq (z_1,\ldots,z_N)$, $z_j = (x_j,v_j)$, is a solution to equation \eqref{eq:NewQN}. Here, $\D_2$ denotes the diagonal of $\T^d$. The quantity $\Hr_{N,\vep}(\uz_N,u)$ is a variant of the modulated energy introduced by Duerinckx and Serfaty \cite[Appendix]{Serfaty2020} for the monokinetic mean-field limit of the system \eqref{eq:NewQN} (i.e. $\vep$ fixed) and has proven to be a good quantity for measuring the distance between the $N$-body dynamics and the limiting dynamics. Our first result, the meat of the article, is a Gronwall-type estimate for the modulated energy, assuming the velocity field $u$ belongs to the H\"older space $C^{1,s}(\T^d)$, for some $0<s<1$.

\begin{thm}
\label{thm:main}
For $d\geq 2$ and $0<s<1$, there exists a constant $C_{d,s}>0$ such that the following holds. Let $u\in L^\infty([0,T]; C^{1,s}(\T^d))$ be a solution to equation \eqref{eq:Eul}, and let $\uz_N$ be a solution to equation \eqref{eq:NewQN} with pairwise distinct initial positions. Then for all $0\leq t\leq T$, we have that
\begin{equation}
\label{eq:main}
\begin{split}
|\Hr_{N,\vep}(\uz_N(t),u(t))| &\leq \Bigg(|\Hr_{N,\vep}(\uz_N(0),u(0))| + C_{d,s}\vep^2 \int_0^t \|u(\tau)\|_{C^{1,s}}^6 d\tau\\
&\ph + \frac{C_{d,s}(1+(\ln N)1_{d=2})}{N^{\frac{2}{d}}\vep^2}\int_0^t \Big(1+\|\nabla u(\tau)\|_{L^\infty} + \vep^2\|\nabla u(\tau)\|_{L^\infty}^2\Big)d\tau\Bigg) \\
&\ph  \times \exp\paren*{C_{d,s}\int_0^t (1+\|\nabla u(\tau)\|_{L^\infty})d\tau}.
\end{split}
\end{equation}
\end{thm}

Our next result, a corollary of \cref{thm:main}, shows that if $\vep=\vep(N)$ is such that the right-hand side of \eqref{eq:main} vanishes as $N\rightarrow\infty$, then $\Hr_{N,\vep}(\uz_N(t),u(t))$ converges to zero on the lifespan of the solution $u$. Moreover, the empirical measure \eqref{eq:EM} of the system \eqref{eq:NewQN} converges to the measure \eqref{eq:limEM} in the space $\M(\T^d\times\R^d)$ of finite Borel measures equipped with the weak-* topology.

\begin{cor}
\label{cor:main}
Let $u\in L^\infty([0,T]; C^{1,s}(\T^d))$ be a solution to equation \eqref{eq:Eul}, and let $\uz_N$ be a sequence of solutions to equation \eqref{eq:NewQN} with pairwise distinct initial positions. If 
\begin{equation}
\label{eq:statasmp}
\vep\xrightarrow[N\rightarrow\infty]{}0, \quad \frac{\vep^2 N^{\frac{2}{d}}}{1+(\ln N)1_{d=2}} \xrightarrow[N\rightarrow\infty]{} \infty, \quad \text{and} \quad \Hr_{N,\vep}(\uz_N(0),u(0)) \xrightarrow[N\rightarrow\infty]{} 0, 
\end{equation}
then
\begin{equation}
\frac{1}{N}\sum_{i=1}^N \d_{z_i(t)} \xrightharpoonup[N\rightarrow \infty]{*} \d(v-u(t,x))dxdv \quad \text{in $\M(\T^d\times\R^d)$ uniformly on $[0,T]$.}
\end{equation}
\end{cor}

We record some remarks about the assumptions in the statements of \cref{thm:main} and \cref{cor:main}.
\begin{remark}
\label{rem:reg}
The incompressible Euler equation is known to be well-posed in the space $C^{1,s}(\T^d)$; for example, see \cite[Proposition 7.16]{BCD2011}). We have not optimized the regularity requirement for $u$ in this article, as our focus is on improving the range of $\theta$ for the validity of the Euler equation as a supercritical mean-field limit. With a bit more work and at the expense of worsening the rate of convergence in $N$, it is possible to relax the regularity assumption to $u\in B_{\infty,1}^1(\T^d)$ (see \eqref{def:Bes} for definition), which is a scaling-critical space for the well-posedness of the equation. It is an interesting mathematical problem whether one can allow for weak solutions $u$, in the spirit of the author's prior work \cite{Rosenzweig2020_PVMF} (see also \cite{Rosenzweig2020_CouMF, Rosenzweig2020_SPVMF}) on the point vortex approximation in dimension $d=2$. We plan to address this question in a separate work.
\end{remark}

\begin{remark}
\label{rem:stat}
Using that $\g$ has zero mean on $\T^d$, we have
\begin{equation}
\frac{1}{\vep^2}\Fr_N(\ux_N(0), 1) = \frac{1}{\vep^{2} N^2}\sum_{1\leq i\neq j\leq N}\g(x_i(0)-x_j(0)).
\end{equation}
Consequently, if the preceding right-hand side converges to zero as $\vep\rightarrow 0$ and $N\rightarrow\infty$, and
\begin{equation}
\frac{1}{N}\sum_{i=1}^N \d_{x_i(0)} \xrightharpoonup[N\rightarrow\infty]{*} 1,
\end{equation}
then $\vep^{-2}\Fr_{N}(\ux_N(0), 1+\vep^2\Uu(0))$ tends to zero. As noted in \cite{HkI2020}, given a sequence $\eta_N\rightarrow 0$ as $N\rightarrow\infty$, one can choose initial velocities $v_i(0)$ such that $|v_i(0)-u(0,x_i(0))|\leq \eta_N$. It then follows that $\Hr_{N,\vep}(\uz_N(0), u(0)) \rightarrow 0$ as $N\rightarrow\infty$ and $\vep\rightarrow 0$. 

The assumption \eqref{eq:statasmp} is statistically relevant for reasons as follows. Suppose that for each $N\geq 1$, the initial positions $x_{1,N}^0,\ldots,x_{N,N}^0$ are independent random points in $\T^d$ with uniform law $\mu \equiv 1$. Then 
\begin{equation}
\frac{1}{\vep^2}\E\paren*{\Fr_{N}(\ux_N(0), 1}) = 0.
\end{equation}
\cref{rem:MEnneg} implies that there is a constant $C_d>0$ such that
\begin{equation}
|\Fr_N(\ux_N(0),1)| \leq \Fr_N(\ux_N(0),1) + \frac{C_d}{N}(|\ln N| 1_{d=2} + N^{1-\frac{2}{d}}1_{d\geq 3}),
\end{equation}
which in turn implies that
\begin{equation}
\frac{1}{\vep^2}\E\paren*{|\Fr_{N}(\ux_N(0), 1)|} \leq \frac{C_d}{\vep^2 N}(|\ln N| 1_{d=2} + N^{1-\frac{2}{d}}1_{d\geq 3}).
\end{equation}
Evidently, the right-hand side tends to zero if $\vep\rightarrow 0$ sufficiently slowly so that $\vep^2 N/\ln N\rightarrow \infty$, if $d=2$, and $\vep N^{1/d} \rightarrow\infty$, if $d\geq 3$, as $N\rightarrow\infty$. 
\end{remark}

Finally, let us address the sharpness of our range \eqref{eq:thrng} for $\theta$. If we assume that the additive error terms in \cref{prop:MElb} below give the optimal scaling for $\Fr_N(\cdot,\cdot)$ as $N\rightarrow\infty$,\footnote{Serfaty has pointed out to the author that in the static case with a confining potential, this optimal scaling indeed holds.} then we expect a necessary condition is that
\begin{equation}
\begin{cases}
\frac{\vep^2 N}{\ln N} \xrightarrow[N\rightarrow\infty]{} \infty, & {d=2} \\
\vep N^{1/d} \xrightarrow[N\rightarrow\infty]{} \infty, & {d\geq 3}.
\end{cases}
\end{equation}
Recalling that $\vep^2 N = N^\theta$, the preceding implies the range $1-\frac{2}{d}<\theta < 1$, which is precisely what we obtain. Thus, we believe the scaling relationship assumed in \cref{cor:main} is optimal. What should be the effective limiting dynamics at and below the threshold $\theta = 1-\frac{2}{d}$ is a very interesting question to which we currently do not know the answer.

\subsection{Comments on the proof}
\label{ssec:introPf}
Let us now briefly discuss the proof of \cref{thm:main}. \cref{cor:main} follows from \cref{thm:main} by relatively standard arguments. At a high level, our proof is inspired by that of Han-Kwan and Iacobelli \cite{HkI2020}, which is based on a Gronwall estimate for the modulated energy $\Hr_{N,\vep}(\uz_N(t),u(t))$ introduced in \eqref{eq:MEdef}. This quantity is quite similar to the modulated energy used by Duerinckx and Serfaty \cite[Appendix]{Serfaty2020} to derive the pressureless Euler-Poisson equation as the mean-field limit of the system \eqref{eq:New} in the monokinetic regime. One can also view it as an $N$-particle version of Brenier's modulated energy \cite{Brenier2000} for the quasineutral limit of Vlasov-Poisson. A key difference, though, compared to these prior works is the presence of the corrector $\vep^2 \Uu$, which leads to good structure in the temporal derivative equation for $\Hr_{N,\vep}(\uz_N(t),u(t))$. To obtain the improved $\theta$ range \eqref{eq:thrng}, though, we need to perform a more sophisticated analysis of the modulated energy to measure the balance of error terms. And for this, we draw on recent work of Serfaty \cite{Serfaty2021} on the fluctuations of Coulomb gasses at arbitrary temperature.

As shown by Han-Kwan and Iacobelli \cite[Section 2]{HkI2020}, $\Hr_{N,\vep}(\uz_N(t),u(t))$ satisfies the equation
\begin{equation}
\begin{split}
\frac{d}{dt}\Hr_{N,\vep}(\uz_N(t),u(t)) = \Te_1+\cdots+\Te_4,
\end{split}
\end{equation}
where
\begin{align}
\Te_1 &= -\frac{1}{N}\sum_{i=1}^N (u(t,x_i(t))-v_i(t))^{\otimes 2} : \nabla u(t,x_i(t)), \label{eq:MET1def}\\
\Te_2 &= \frac{1}{2\vep^2}\int_{(\T^d)^2\setminus\D_2} \paren*{u(t,x)-u(t,y)}\cdot\nabla\g(x-y)d(\frac{1}{N}\sum_{i=1}^N\d_{x_i(t)}-1-\vep^2 \Uu(t))^{\otimes 2}(x,y) \label{eq:MET2def}\\
\Te_3 &= -\int_{\T^d}\nabla\cdot\Dm^{-2}(u\U)(t,x)d(\frac{1}{N}\sum_{i=1}^N\d_{x_i(t)}-1-\vep^2\Uu(t))(x) \label{eq:MET3def}\\
\Te_4 &= -\int_{\T^d} (\p_t p )(t,x)d(\frac{1}{N}\sum_{i=1}^N\d_{x_i(t)}-1-\vep^2\Uu(t))(x). \label{eq:MET4def}
\end{align}
Above, $:$ denotes the Frobenius inner product for $d\times d$ matrices. Due to the prefactor of $\vep^{-2}$, $\Te_2$ is the most challenging term. Therefore, we concentrate on it for the purposes of this discussion.

$\Te_2$ has the structure of a commutator which has been \emph{renormalized} through the excision of the diagonal and has been averaged against the measure
\begin{equation}
\label{eq:commmes}
\frac{1}{N}\sum_{i=1}^N\d_{x_i} - 1-\vep^2\Uu.
\end{equation}
Indeed, ignoring the excision of the diagonal, the integration in $y$ corresponds to the commutator
\begin{equation}
\comm{u}{\nabla\Dm^{-2}}\paren*{\frac{1}{N}\sum_{i=1}^N\d_{x_i} - 1-\vep^2\Uu}.
\end{equation}
$\Te_2$ may also be written as the divergence of the stress-energy tensor of the potential of \eqref{eq:commmes} integrated against the vector field $u$. Of course, the preceding considerations are completely formal because we have excised the diagonal in $\Te_2$ and because the Dirac mass is too singular for either the commutator or stress-energy tensor to make sense. However, Serfaty showed \cite[Proposition 1.1]{Serfaty2020} using a smearing procedure for the Dirac masses (see \cref{ssec:MEbas}) that one can add back the diagonal and use this stress-energy tensor idea to control expressions of the form $\Te_2$ up to additive error terms quantifiably small as $N\rightarrow\infty$. Using this result, Han-Kwan and Iacobelli obtained an estimate for $\Te_2$ in terms of the quantity \eqref{eq:MEpdef}. The source of the restriction $\vep N^{\frac{1}{d(d+1)}}\rightarrow \infty$ as $N\rightarrow\infty$, equivalently $1-\frac{2}{d(d+1)}<\theta<1$, is precisely the estimate for the aforementioned additive error terms that \cite[Proposition 1.1]{Serfaty2020} gives.

In \cite{Rosenzweig2020_PVMF} (see also \cite{Rosenzweig2020_CouMF, Rosenzweig2020_SPVMF}), the author complemented Serfaty's smearing procedure with the new idea of tracking the relative size of these additive terms through a series of small, possibly time-dependent parameters. One keeps these parameters unspecified until the conclusion of the proof, where they are chosen to optimize the balance of all the error terms accumulated to estimate expressions of the form $\Te_2$. We could implement a similar idea in this article, which ultimately gives a better estimate for the additive error terms in Serfaty's original smearing procedure. But this would only yield the range \eqref{eq:thrng} in dimension $d=2$. In order to obtain the stated range \eqref{eq:thrng} in all dimensions $d\geq 2$, we instead use an improved estimate (see \cref{prop:com} in \cref{ssec:MEcomm} below for details) for expressions of the form \eqref{eq:MET2def} recently obtained by Serfaty \cite{Serfaty2021}, which yields the (believed) sharp $N$-dependence for the additive errors.

\begin{remark}
\label{rem:gyro}
Han-Kwan and Iacobelli \cite[Theorem 1.2]{HkI2020} also consider the combined mean-field and \emph{gyrokinetic} limit for particles interacting in the presence of a strong magnetic field in dimension $d=2$ using an analogous modulated-energy approach. The $N$-body problem is now
\begin{equation}
\begin{cases}
\vep \dot{x}_i = v_i \\
\vep\dot{v}_i = -\frac{1}{N}\sum_{{1\leq j\leq N}\atop{j\neq i}}\nabla\g(x_i-x_j) + \frac{v_i^\perp}{\vep},
\end{cases}
\end{equation}
and the limiting behavior is again governed by the incompressible Euler equation. We expect our analysis to improve their scaling restriction $\vep N^{1/6}\rightarrow \infty$ as $N\rightarrow\infty$ to $\vep (N/\ln N)^{1/2} \rightarrow\infty$ as $N\rightarrow\infty$.
\end{remark}

\subsection{Organization of article}
\label{ssec:introOrg}
We briefly comment on the organization of the body of the article. In \cref{sec:pre}, we introduce basic notation, function spaces, and properties of the Coulomb potential used without further comment throughout the article. In \cref{sec:ME}, we review Serfaty's smearing procedure and properties of the modulated potential energy $\Fr_{N}(\cdot,\cdot)$. Since the existing literature almost exclusively considers the case of $\R^d$, we focus on the modifications necessary for the periodic setting, as they do not seem to be presented in the literature. We also give in \cref{ssec:MEcomm} the proof of our ``commutator'' estimate for expressions of the form \eqref{eq:MET2def}, which is the workhorse of this article. Finally, in \cref{sec:MR}, we prove our main results, \cref{thm:main} and \cref{cor:main}.

\subsection{Acknowledgments}
The author thanks Sylvia Serfaty for valuable discussion related to this project, in particular explaining to him the expected optimal scaling for the modulated potential energy and for informing him of the result \cite[Proposition 4.3]{Serfaty2021}, which is crucial to obtaining the full range \eqref{eq:thrng} in dimensions $d\geq 3$. The author also thanks Willie Wong for helpful comments on the multiplicative algebra properties of Besov spaces, which informed the regularity assumptions on the solution $u$ to \eqref{eq:Eul}. The author gratefully acknowledges funding from the Simons Foundation through the Simons Collaboration on Wave Turbulence.

\section{Preliminaries}
\label{sec:pre}
In this section, we introduce the basic notation of the article, as well as the relevant function spaces, and some facts from harmonic analysis to which we frequently appeal.

\subsection{Notation}
\label{ssec:preNot}
Given nonnegative quantities $A$ and $B$, we write $A\lesssim B$ if there exists a constant $C>0$, independent of $A$ and $B$, such that $A\leq CB$. If $A \lesssim B$ and $B\lesssim A$, we write $A\sim B$. To emphasize the dependence of the constant $C$ on some parameter $p$, we sometimes write $A\lesssim_p B$ or $A\sim_p B$.

We denote the natural numbers excluding zero by $\N$ and including zero by $\N_0$. Similarly, we denote the nonnegative real numbers by $\R_{\geq 0}$ and the positive real numbers by $\R_+$. Given $N\in\N$ and points $x_{1,N},\ldots,x_{N,N}$ in some set $X$, we will write $\ux_N$ to denote the $N$-tuple $(x_{1,N},\ldots,x_{N,N})$. We define the generalized diagonal $\Delta_N$ of the Cartesian product $X^N$ to be the set
\begin{equation}
\Delta_N \coloneqq \{(x_1,\ldots,x_N) \in X^N : x_i=x_j \text{ for some $i\neq j$}\}.
\end{equation}
Given $x\in\T^d$ and $r>0$, we denote the ball and sphere centered at $x$ of radius $r$ by $B(x,r)$ and $\p B(x,r)$, respectively. Given a function $f$, we denote the support of $f$ by $\supp f$. 

We denote the space of complex-valued Borel measures on $\T^d$ by $\M(\T^d)$. We denote the subspace of probability measures (i.e. elements $\mu\in\M(\T^d)$ with $\mu\geq 0$ and $\mu(\T^d)=1$) by $\P(\T^d)$. When $\mu$ is in fact absolutely continuous with respect to Lebesgue measure on $\T^d$, we shall abuse notation by writing $\mu$ for both the measure and its density function. We denote the Banach space of complex-valued continuous functions on $\T^d$ by $C(\T^d)$ equipped with the uniform norm $\|\cdot\|_{\infty}$. More generally, we denote the Banach space of $k$-times continuously differentiable functions with bounded derivatives up to order $k$ by $C^k(\T^d)$ equipped with the natural norm, and we define $C^\infty \coloneqq \bigcap_{k=1}^\infty C^k$. We denote the subspace of smooth functions with compact support by $C_c^\infty(\T^d)$, and use the subscript $0$ to indicate functions vanishing at infinity. For $p\in [1,\infty]$ and $\Omega\subset\T^d$, we define $L^p(\Omega)$ to be the usual Banach space equipped with the norm
\begin{equation}
\|f\|_{L^p(\Omega)} \coloneqq \paren*{\int_\Omega |f(x)|^p dx}^{1/p}
\end{equation}
with the obvious modification if $p=\infty$. When $f: \Omega\rightarrow X$ takes values in some Banach space $(X,\|\cdot\|_{X})$, we shall write $\|f\|_{L^p(\Omega;X)}$.

Our convention for Fourier coefficients is
\begin{align}
\wh{f}(k) &\coloneqq \int_{\T^d}f(x)e^{-ik\cdot x}dx , \qquad k\in\Z^d\\
f(x) &= (2\pi)^{-d} \sum_{k\in\Z^d}\wh{f}(k)e^{ik\cdot x}, \qquad x\in \T^d.
\end{align}

\subsection{Besov spaces}
\label{ssec:preBes}
Let $\Dc'(\T^d)$ denote the dual of the space $C^\infty(\T^d)$. As is well-known (e.g. see \cite[Section 3.2]{ST1987}), elements of $\Dc'(\T^d)$ are precisely those tempered distributions on $\R^d$ for which there exists a sequence of coefficients $\{a_k\}_{k\in\Z^d}$ with $|a_k| \lesssim \jp{k}^m$ for some $m\in\N$, such that the series
\begin{equation}
\sum_{k\in\Z^d} a_k e^{ik\cdot x}
\end{equation}
converges in $\Sc'(\R^d)$.

To introduce the Besov scale of function spaces, we introduce a Littlewood-Paley partition of unity as follows. Let $\phi\in C_{c}^{\infty}(\R^d)$ be a radial, nonincreasing function, such that $0\leq \phi\leq 1$ and
\begin{equation}
\label{eq:idphi}
\phi(x)=
\begin{cases}
1, & {|x|\leq 1}\\
0, & {|x|>3/2}.
\end{cases}
\end{equation}
Define
\begin{align}
1&=\phi(x)+\sum_{j=1}^{\infty}[\phi(2^{-j}x)-\phi(2^{-j+1}x)] \eqqcolon \phi(x)+\sum_{j=1}^{\infty}\psi_{j}(x), \qquad \forall x\in\R^d.
\end{align}
Given a distribution $f\in\Dc'(\T^d)$, we define the Littlewood-Paley projections
\begin{equation}
\begin{split}
P_0 f(x) &\coloneqq (2\pi)^{-d}\sum_{k\in\Z^d}\phi(k)\wh{f}(k)e^{ik\cdot x}, \\
P_j f(x) &\coloneqq (2\pi)^{-d}\sum_{k\in\Z^d} \psi_j(k)\wh{f}(k)e^{ik\cdot x}, \qquad j\geq 1,
\end{split}
\end{equation}
and $P_j f\coloneqq 0$ for $j\leq -1$. Evidently, each of the projections $P_j f\in C^\infty(\T^d)$. 

\begin{mydef}
\label{def:Bes}
Let $s\in\R$ and $1\leq p,q\leq\infty$. We define the inhomogeneous Besov space $B_{p,q}^s(\T^d)$ to be the space of $f\in\Dc'(\T^d)$ such that
\begin{equation}
\|f\|_{B_{p,q}^s(\T^d)} \coloneqq \paren*{\sum_{j=0}^\infty 2^{jqs}\|P_jf\|_{L^p(\T^d)}^q}^{1/q} < \infty.
\end{equation}
\end{mydef}

\begin{remark}
The Besov scale includes a number of well-known function spaces (see \cite[Section 3.5]{ST1987} for full details). For example, if $s\geq 0$ is not an integer, then $B_{\infty,\infty}^s = C^{\lfloor{s}\rfloor, s-\lfloor{s}\rfloor}$, while if $s$ is an integer, we have the inclusions
\begin{equation}
B_{\infty,1}^s \subset C^s \subset B_{\infty,\infty}^s.
\end{equation}
Similarly, for any $1\leq p<\infty$, we have the inclusions
\begin{equation}
B_{p,1}^0 \subset L^p \subset B_{p,\infty}^0.
\end{equation}
If $p=q=2$, then $B_{2,2}^s$ coincides with the usual Sobolev space $H^s$.
\end{remark}

The next two results on embeddings in the Besov scale are classic. The reader may consult \cite[Sections 2.7, 2.8]{BCD2011} for the omitted details.
\begin{lemma}
\label{lem:besemb}
Let $1\leq p_1\leq p_2\leq \infty$, $1\leq q_1\leq q_2\leq\infty$, and $s\in\R$. Then $B_{p_1,q_1}^s(\T^d)$ continuously embeds in $B_{p_2,q_2}^{s-\frac{d}{p_1}+\frac{d}{p_2}}(\T^d)$. Additionally, if $s\geq 0$, then $B_{\infty,1}^s(\T^d)$ continuously embeds in $C^s(\T^d)$.
\end{lemma}

\begin{lemma}
\label{prop:Bony}
For $s>0$ and $1\leq p,q\leq\infty$, the space $L^\infty(\T^d)\cap B_{p,q}^s(\T^d)$ is an algebra and
\begin{equation}
\|fg\|_{B_{p,q}^s} \lesssim_{s,d} \|f\|_{L^\infty}\|g\|_{B_{p,q}^s} + \|g\|_{L^\infty}\|f\|_{B_{p,q}^s}.
\end{equation}
\end{lemma}

\subsection{The Coulomb potential}
We recall from the introduction that $\g$ denotes the element of $\Dc'(\T^d)$
\begin{equation}
\label{eq:gFS}
\frac{1}{(2\pi)^d}\sum_{{k\in\Z^d} : {k\neq 0}} \frac{e^{ik\cdot x}}{|k|^2},
\end{equation}
which is, in fact, an element of $C^\infty(\T^d\setminus\{0\})$ satisfying
\begin{equation}
-\D\g = \d_{0} - 1.
\end{equation}
Let $\g_{\R^d}$ denote the Coulomb potential on $\R^d$:
\begin{equation}
\label{eq:gRd}
\g_{\R^d}(x) \coloneqq  \begin{cases}
-(\ln|x|)/2\pi, & {d=2} \\
c_d|x|^{-d+2}, & {d\geq 3},
\end{cases}
\end{equation}
where the normalizing constant $c_d = \frac{1}{d(d-2)|B(0,1)|}$. As is well-known, there exists a function $\g_{loc} \in C^\infty(\ol{B(0,1/4)})$ such that 
\begin{equation}
\label{eq:gasy}
\g(x) = \g_{\R^d}(x) + \g_{loc}(x), \qquad \forall x\in\ol{B(0,1/4)}.
\end{equation}

The next lemma will be quite useful in the proof of \cref{prop:com}.
\begin{lemma}
\label{lem:gnab}
Let $d\geq 2$ and $d<p\leq\infty$. Then for all $f\in L^p(\T^d)$,
\begin{equation}
\|\nabla\Dm^{-2} f\|_{L^\infty} = \|\nabla\g\ast f\|_{L^\infty} \lesssim_{p,d} \|f\|_{L^p}.
\end{equation}
\end{lemma}
\begin{proof}
The first identity is tautological. For $x\in \T^d$, we write
\begin{align}
\int_{\T^d}\nabla\g(x-y)f(y)dy &= \int_{|x-y|\leq 1/4}\nabla\g(x-y)f(y)dy + \int_{|x-y|>1/4}\nabla\g(x-y)f(y)dy \nn\\
&=\int_{|x-y|\leq 1/4}\nabla\g_{\R^d}(x-y)f(y)dy + \int_{|x-y|\leq 1/4} \nabla\g_{loc}(x-y)f(y)dy \nn\\
&\ph+ \int_{|x-y|>1/4}\nabla\g(x-y)f(y)dy.
\end{align}
Since $\g_{loc}$ is $C^\infty$ on $\ol{B(0,1/4)}$ and since $\g$ is $C^\infty$ away from the origin, it follows from H\"older that
\begin{equation}
\int_{|x-y|\leq1/4}|\nabla\g_{loc}(x-y)f(y)|dy + \int_{|x-y|>1/4}|\nabla\g(x-y)f(y)|dy \lesssim_{d,p} \|f\|_{L^p}.
\end{equation}
Again applying H\"older,
\begin{equation}
\int_{|x-y|\leq 1/4}|\nabla\g_{\R^d}(x-y)f(y)|dy \leq \|\nabla\g_{\R^d}\|_{L^{p'}(B(0,1/4))} \|f\|_{L^p}.
\end{equation}
Since $(d-1)p'<d$ by assumption that $p>d$, we see that the proof is complete.
\end{proof}

\section{The modulated energy}
\label{sec:ME}
In this section, we discuss the properties of the potential part of the $N$-particle modulated energy,
\begin{equation}
\Fr_{N}(\ux_N,\mu) = \int_{(\T^d)^2\setminus\D_2} \g(x-y) d(\frac{1}{N}\sum_{i=1}^N\d_{x_i} - \mu)^{\otimes 2}(x,y).
\end{equation}
Most of the results in this section have been established in one form or another in the works using the modulated-energy method, such as \cite{Serfaty2017, Duerinckx2016, Serfaty2020, Rosenzweig2020_PVMF, Rosenzweig2020_CouMF, Rosenzweig2020_SPVMF}. Therefore, we shall be somewhat terse in our remarks, focusing on the proof modifications necessary to adapt the comparable result in literature to the periodic setting, where some subtleties arise due to the fact that the potential $\g$ is no longer explicit.

\subsection{Smearing and truncation}
\label{ssec:MEbas}
Given $\eta\in (0,1/4)$, let us truncate the potential $\g$ by defining
\begin{equation}
\tl{\g}_\eta(x) \coloneqq \begin{cases} \g(x), & x\in \T^d\setminus B(0,\eta) \\ \g_{\R^d}(\eta) + \g_{loc}(x), & x\in B(0,\eta). \end{cases}
\end{equation}
To obtain a function with zero mean, we now consider
\begin{equation}
\label{eq:geta}
\g_\eta(x) \coloneqq \tl{\g}_\eta(x) - {\int_{\T^d}\tl{\g}_\eta(x)dx} \eqqcolon \tl{\g}_\eta(x) - c_{\g,\eta}.
\end{equation}
We now introduce a distribution $\d_0^{(\eta)}$ by setting
\begin{equation}
\d_0^{(\eta)} - 1\coloneqq -\D\g_\eta.
\end{equation}

\begin{lemma}
Let $\sigma_{\p B(0,\eta)}$ denote the uniform probability measure on the sphere $\p B(0,\eta)$. Then
\begin{equation}
\label{eq:deta}
\d_0^{(\eta)} = \sigma_{\p B(0,\eta)} \quad \text{in} \quad \mathcal{D}'(\T^d).
\end{equation}
\end{lemma}
\begin{proof}
Let $\varphi\in C^\infty(\T^d)$ be a test function. Integrating by parts twice,
\begin{align}
\int_{\T^d}-\D\varphi(x)\g_\eta(x)dx &= \int_{|x|<\eta}-\D\varphi(x)\paren*{\g_{\R^d}(\eta)+\g_{loc}(x)}dx + \int_{|x|\geq\eta}-\D\varphi(x)\g_\eta(x)dx \nn\\
&=\int_{|x|<\eta}\na\varphi(x)\cdot\na\g_{loc}(x)dx +  \int_{|x|\geq\eta}\nabla\varphi(x)\cdot\nabla\g(x)dx \nn\\
&=\int_{|x|=\eta}\varphi(x)\na\g_{loc}(x)\cdot\frac{x}{|x|}d\H^{d-1}(x) - \int_{|x|<\eta}\varphi(x)\D\g_{loc}(x)dx \nn\\
&\ph + \int_{|x|=\eta}\varphi(x)\nabla\g(x)\cdot\frac{-x}{|x|}d\H^{d-1}(x) - \int_{|x|\geq\eta}\varphi(x)\D\g(x)dx,
\end{align}
where $\H^{d-1}$ denotes the $(d-1)$-dimensional Hausdorff measure. Observe that for $|x|=\eta$,
\begin{equation}
\na\g_{loc}(x)\cdot\frac{x}{|x|} - \na\g(x)\cdot\frac{x}{|x|} = -\na\g_{\R^d}(x)\cdot\frac{x}{|x|} = (d-2)c_d\frac{x}{|x|^{d}}\cdot\frac{x}{|x|} = \frac{1}{d|B(0,1)|\eta^{d-1}}.
\end{equation}
Since $\H^{d-1}(\p B(0,\eta)) = d|B(0,1)|\eta^{d-1}$, it follows that
\begin{equation}
\begin{split}
&\int_{|x|=\eta}\varphi(x)\na\g_{loc}(x)\cdot\frac{x}{|x|}d\H^{d-1}(x)+ \int_{|x|=\eta}\varphi(x)\nabla\g(x)\cdot\frac{-x}{|x|}d\H^{d-1}(x) \\
&= \frac{1}{\H^{d-1}(\p B(0,\eta))}\int_{|x|=\eta}\varphi(x)d\H^{d-1}(x).
\end{split}
\end{equation}
Since $-\D\g(x) = -1$ for $|x|>0$ and $-\D\g_{loc}(x) = -1$ for $|x|\leq 1/4$, we also find that
\begin{equation}
- \int_{|x|<\eta}\varphi(x)\D\g_{loc}(x)dx- \int_{|x|\geq\eta}\varphi(x)\D\g(x)dx = -\int_{\T^d}\varphi(x)dx.
\end{equation}
Thus, we have shown that
\begin{equation}
\ipp{\varphi,-\D\g_\eta} = \ipp{\varphi,\sigma_{\p B(0,\eta)}-1}.
\end{equation}
Since $\varphi\in C^\infty(\T^d)$ was arbitrary, the conclusion of the lemma follows immediately. 
\end{proof}

Using the definition \eqref{eq:deta} of $\d_0^{(\eta)}$, we have the identity
\begin{align}
\label{eq:idgconv}
\g\ast(\d_0-\d_0^{(\eta)}) = \g\ast\d_0 - \g\ast(-\D\g_\eta + 1) = \g - (\d_0-1)\ast\g_\eta = \g-\g_\eta \eqqcolon \f_{\eta}.
\end{align}
where we use that $\g, \g_\eta$ both have mean zero. For later use, let us now analyze $\f_{\eta}$. The definition \eqref{eq:geta} of $\g_\eta$ implies that
\begin{equation}
\label{eq:idgdiff}
\f_\eta(x) = \begin{cases} 0, & |x|\geq\eta \\ \g_{\R^d}(x)-\g_{\R^d}(\eta) + c_{\g,\eta}, & {|x|<\eta}. \end{cases}
\end{equation}
We also have the gradient identity
\begin{equation}
\label{eq:idgddiff}
\nabla\f_\eta(x) = \nabla\g_{\R^d}(x)1_{\leq \eta}(x).
\end{equation}
We estimate the normalizing constant $c_{\g,\eta}$ with the next lemma.

\begin{lemma}
\label{lem:cg}
For all $0<\eta<1/4$, it holds that
\begin{equation}
|c_{\g,\eta}| \lesssim_d \eta^2.
\end{equation}
\end{lemma}
\begin{proof}
Since $\g$ has zero mean on $\T^d$, it follows that
\begin{align}
\label{eq:idcgeta}
c_{\g,\eta} &= \int_{|x|\geq\eta}\g(x)dx + \int_{|x|<\eta}\paren*{\g_{\R^d}(\eta) + \g_{loc}(x)}dx \nn\\
&=-\int_{|x|<\eta}\g(x)dx + \int_{|x|<\eta}\paren*{\g_{\R^d}(\eta)+\g_{loc}(x)}dx \nn\\
&=\int_{|x|<\eta}\paren*{\g_{\R^d}(\eta)-\g_{\R^d}(x)}dx \nn\\
&=\g_{\R^d}(\eta)|B(0,\eta)| - \int_{|x|<\eta}\g_{\R^d}(x)dx.
\end{align}
Let us compute the second term. First, suppose that $d=2$. Unpacking the definition of $\g_{\R^2}$ and using polar coordinates, we find that
\begin{align}
\int_{|x|<\eta} \g_{\R^2}(x)dx = -\int_0^\eta r\ln r dr = -\frac{\eta^2\ln\eta}{2} +\frac{\eta^2}{4}.
\end{align}
Next, suppose that $d\geq 3$. Proceeding similarly, we find that
\begin{align}
\int_{|x|<\eta} \g_{\R^d}(x)dx = c_d \H^{d-1}(\p B(0,1))\int_0^\eta r^{d-1}r^{2-d}dr = \frac{c_d \eta^2 \H^{d-1}(\p B(0,1))}{2}.
\end{align}
Since $\H^{d-1}(\p B(0,1)) = d|B(0,1)|$, it follows that
\begin{equation}
\label{eq:cgeta}
c_{\g,\eta} = \g_{\R^d}(\eta)|B(0,\eta)| - \int_{|x|<\eta}\g_{\R^d}(x)dx =
\begin{cases}
-\frac{\eta^2}{4}, & {d=2} \\ \frac{\eta^2(d^{-1}-2^{-1})}{d-2}, &{d\geq 3}.
\end{cases}
\end{equation}
\end{proof}

Next, we use the preceding computation in order to estimate the $L^p$ norms of $\f_\eta, \nabla\f_\eta$, for appropriate $p$.

\begin{lemma}
\label{lem:gLp}
For $d\geq 2$ and $1\leq p< d/(d-2)$, it holds that
\begin{equation}
\forall 0<\eta<1/4, \qquad \|\f_\eta\|_{L^p(\T^d)} \lesssim_{d,p}
\begin{cases}
\eta^{\frac{2}{p}}|\ln\eta|, & {d=2} \\
\eta^{2-\frac{d(p-1)}{p}}, & {d\geq 3};
\end{cases}
\end{equation}
and for $1\leq p<d/(d-1)$, it holds that
\begin{equation}
\|\nabla\f_\eta\|_{L^p(\T^d)} \lesssim_{d,p} \eta^{1-\frac{d(p-1)}{p}}.
\end{equation}
\end{lemma}
\begin{proof}
By the triangle inequality and \cref{lem:cg},
\begin{equation}
\|\f_\eta\|_{L^p} \lesssim_d \|\g_{\R^d}-\g_{\R^d}(\eta)\|_{L^p} + \eta^2.
\end{equation}
For $d=2$, we change to polar coordinates to obtain
\begin{align}
\|\g_{\R^2}-\g_{\R^2}(\eta)\|_{L^p} &\lesssim \paren*{\int_0^\eta r(-\ln r + \ln\eta)^p dr}^{1/p} \lesssim_p \eta^{2/p}|\ln \eta|.
\end{align}
Similarly, if $d\geq 3$, then
\begin{align}
\|\g_{\R^d}-\g_{\R^d}(\eta)\|_{L^p} &\lesssim_d \paren*{\int_0^\eta r^{d-1}(r^{2-d} - \eta^{2-d})^pdr}^{1/p} \lesssim_{p,d} \eta^{2-\frac{d(p-1)}{p}},
\end{align}
provided that $1\leq p<d/(d-2)$. After a little bookkeeping, we arrive at the first assertion. The second assertion follows from the identity \eqref{eq:idgddiff} and proceeding as before.
\end{proof} 

We also introduce the notation\footnote{The reader should note that our notation corresponds to $N^{-2}$ times the notation $H_{N,\ueta_N}^{\mu,\ux_N}$ used in \cite{Rosenzweig2020_PVMF}.}
\begin{equation}
H_{N,\ueta_N}^{\mu,\ux_N}(x) \coloneqq (\g\ast(\frac{1}{N}\sum_{i=1}^N \d_{x_i}^{(\eta_i)}-\mu))(x) = \Dm^{-2}(\frac{1}{N}\sum_{i=1}^N \d_{x_i}^{(\eta_i)}-\mu)(x)
\end{equation}
to denote the $\g$-potential of the difference $\frac{1}{N}\sum_{i=1}^N\d_{x_i}^{(\eta_i)} - \mu$. Observe from Plancherel's theorem that
\begin{equation}
\label{eq:gradH}
\int_{\T^d} |\nabla H_{N,\ueta_N}^{\mu,\ux_N}(x)|^2 dx = \int_{(\T^d)^2}\g(x-y)d(\frac{1}{N}\sum_{i=1}^N \d_{x_i}^{(\eta_i)} - \mu)^{\otimes 2}(x,y).
\end{equation}
In particular, since $\mu\in L^p(\T^d)$, for $p>d/2$, implies that $\mu$ has finite Coulomb energy and $\g\ast\mu$ is H\"older continuous, the left-hand side is also finite under this assumption.

\begin{prop}
\label{prop:MElb}
The exists a constant $C_d>0$, such that for any $\ux_N\in (\T^d)^N\setminus\D_N$, $\mu\in L^\infty(\T^d)$, and $0<\eta_1,\ldots,\eta_N < 1/8$, it holds that
\begin{equation}
\begin{split}
\frac{1}{N^2}\sum_{{1\leq i\neq j\leq N}}\paren*{\g_{\R^d}(x_i-x_j) - \g_{\R^d}(\eta_i)}_{+} &\leq \Fr_N(\ux_N,\mu) + \frac{C_d}{N^2}\sum_{j=1}^N \paren*{|\ln\eta_j|1_{d=2} + \eta_j^{2-d} 1_{d\geq 3}} \\
&\ph - \int_{\T^d} |\nabla H_{N,\ueta_N}^{\mu,\ux_N}(x)|^2 dx  + \frac{C_d(1+\|\mu\|_{L^\infty})}{N}\sum_{j=1}^N \eta_j^{2},
\end{split}
\end{equation}
where $(\cdot)_+ \coloneqq \max\{0,\cdot\}$.
\end{prop}
\begin{proof}
See the proof of \cite[Proposition 3.3]{Serfaty2020}.
\end{proof}

\begin{remark}
\label{rem:MEnneg}
Since the left-hand side of the inequality in \cref{prop:MElb} is nonnegative, the proposition shows that there is a constant $C_d>0$ such that
\begin{equation}
\Fr_N(\ux_N,\mu) + \frac{C_d}{N^2}\sum_{j=1}^N \paren*{|\ln\eta_j|1_{d=2} + \eta_j^{2-d} 1_{d\geq 3}}  + \frac{C_d(1+\|\mu\|_{L^\infty})}{N}\sum_{j=1}^N \eta_j^{2} \geq 0
\end{equation}
for all choices $0<\eta_1,\ldots,\eta_N < 1/8$. Consequently, up to the additive error detailed above, $\Fr_N$ and $|\Fr_N|$ are comparable.
\end{remark}

A quick corollary of \cref{prop:MElb} is that if we choose the smearing/truncation scales $\ueta_N$ to be comparable to the nearest-neighbor distances of each particle, we can use the modulated energy to control--up to a small additiver error--the $\dot{H}^1$ norm of the smeared potential $H_{N,\ueta_N}^{\mu,\ux_N}$, as well as the self-interaction contribution to the energy.
\begin{cor}
\label{cor:MEct}
Let $d\geq 2$. Then there exists a constant $C_d>0$, such that for any $\ux_N\in (\T^d)^N\setminus\D_N$, $\mu\in L^\infty(\T^d)$, $0<\ep<1/8$, it holds that if
\begin{equation}
\label{eq:ridef}
r_{i,\ep} \coloneqq \min\{\frac{1}{4}\min_{1\leq j\leq N : j\neq i} |x_i-x_j|,\ep\}, \qquad \ur_{N,\ep} \coloneqq (r_{1,\ep},\ldots,r_{N,\ep}),
\end{equation}
then
\begin{equation}
\label{eq:MEHlb}
\begin{split}
\int_{\T^d} |\nabla H_{N,\ur_{N,\ep}}^{\mu,\ux_N}(x)|^2 dx \leq \Fr_{N}(\ux_N,\mu) + \frac{C_d}{N}\paren*{|\ln\ep|1_{d=2} + \ep^{2-d}1_{d\geq 3}} + C_{d}(1+\|\mu\|_{L^\infty})\ep^2,
\end{split}
\end{equation}
\begin{equation}
\label{eq:MEsi}
\begin{split}
\frac{1}{N^2}\sum_{i=1}^N \g_{\R^d}(r_{i,\ep}) \leq \Fr_{N}(\ux_N,\mu) + \frac{C_d}{N}\paren*{|\ln\ep|1_{d=2} + \ep^{2-d}1_{d\geq 3}} + C_{d}(1+\|\mu\|_{L^\infty})\ep^2.
\end{split}
\end{equation}
\end{cor}
\begin{proof}
For the proof of estimates \eqref{eq:MEHlb} and \eqref{eq:MEsi}, see \cite[Corollary 3.6]{Rosenzweig2020_PVMF} for the case $d=2$ and \cite[Corollary 3.6]{Rosenzweig2020_CouMF} for the case $d\geq 3$.
\end{proof}

\begin{remark}
\cref{cor:MEct} is a simple generalization of \cite[Corollary 3.4]{Serfaty2020}, where $\ep$ is chosen to be $N^{-1/d}$, which is the typical interparticle distance. Note that the choice $\ep=N^{-1/d}$ is such that both error terms in \eqref{eq:MEHlb} and \eqref{eq:MEsi} are of the same order in $N$ (up to a $\ln N$ factor if $d=2$). The reason we leave $\ep$ unspecified here and throughout the article until the very end is that doing so makes transparent the significance of the scale $N^{-1/d}$ as the optimal choice. Also, if $d=2$, one can exploit the fact that $\ep$ can be chosen to be $N^{-\beta}$ for arbitrarily large fixed $\beta$, at the cost of making the constant $C_d$ larger. This trick has the benefit of avoiding the more complicated estimate of \cref{prop:com} below, a point on which we comment more at the beginning of \cref{ssec:MEcomm}.
\end{remark}

\subsection{Control of Sobolev norms}
\label{ssec:ME_Sob}
Although $\Fr_N(\cdot,\cdot)$ is not a nonnegative quantity, as noted in \cref{rem:MEnneg}, it nevertheless controls Sobolev norms up to small error terms. The following proposition shows that the difference $\frac{1}{N}\sum_{i=1}^N \d_{x_i} - \mu$ belongs to the dual space $(W^{1,\infty}(\T^d))^*$, which by Sobolev embedding implies that it belongs to $H^s(\T^d)$ for any $s<-\frac{d+2}{2}$.

\begin{prop}
\label{prop:MESob}
For $d\geq 2$, there exists a constant $C_d>0$, such that the following holds: for any $\varphi\in W^{1,\infty}(\T^d)$, $\mu\in L^\infty(\T^d)$, and $\ux_N\in(\T^d)^N\setminus\D_N$,
\begin{equation}
\begin{split}
\left|\int_{\T^d}\varphi(x)d(\frac{1}{N}\sum_{i=1}^N\d_{x_i}-\mu)(x)\right| &\lesssim_d \ep\|\nabla\varphi\|_{L^\infty} + \|\nabla\varphi\|_{L^2}\Bigg(\Fr_N(\ux_N,\mu) + \frac{C_d}{N}\Big(|\ln\ep|1_{d=2} + \ep^{2-d}1_{d\geq 3}\Big) \\
&\ph + C_d(1+\|\mu\|_{L^\infty})\ep^2 \Bigg)^{1/2}
\end{split}
\end{equation}
for all choices $0<\ep<1/8$. 
\end{prop}
\begin{proof}
See the proof of \cite[Proposition 3.6]{Serfaty2020}.
\end{proof}

\subsection{Commutator estimate}
\label{ssec:MEcomm}
As alluded to in \cref{ssec:introPf} of the introduction, a key ingredient in the proof of \cref{thm:main} is essentially an order-1 smoothing estimate for the commutator
\begin{equation}
\comm{v}{\nabla\Dm^{-2}}(\frac{1}{N}\sum_{i=1}^N\d_{x_i}-\mu) \coloneqq v\cdot\nabla\Dm^{-2}(\frac{1}{N}\sum_{i=1}^N\d_{x_i}-\mu) - \nabla\Dm^{-2}\cdot(v(\frac{1}{N}\sum_{i=1}^N\d_{x_i}-\mu)),
\end{equation}
where $v$ is a sufficiently smooth vector field, $\nabla\Dm^{-2}$ is the vector valued Fourier multplier with symbol $\frac{i\xi}{|\xi|^2}$, and $\mu$ is a measure with $L^\infty$ density. More precisely, we need to estimate expressions of the form
\begin{equation}
\label{eq:rcomm}
\int_{(\T^d)^2\setminus\D_2} \nabla\g(x-y)\cdot(v(x)-v(y))d(\frac{1}{N}\sum_{i=1}^N \d_{x_i}-\mu)^{\otimes 2}(x,y)
\end{equation}
in terms of the quantity $\Fr_N(\ux_N,\mu)$ and some small $N$-dependent error terms. Estimates of these kind were originally proved by Serfaty \cite[Proposition 1.1]{Serfaty2020} using the formalism of the stress-energy tensor in the setting of $\R^d$, but the proof easily adapts to setting of $\T^d$. The author later made the commutator structure transparent in \cite[Proposition 4.1]{Rosenzweig2020_SPVMF} and obtained the (believed) sharp $N$-dependence for the error terms. This latter result would suffice in dimension $d=2$, but for $d\geq 3$, it would still yield a sub-optimal range for $\theta$ compared to \eqref{eq:thrng}, but still better than the range \eqref{eq:th_res} of Han-Kwan and Iacobelli \cite{HkI2020}. To obtain the full range \eqref{eq:thrng} in all dimensions $d\geq 2$, we instead use a recent result of Serfaty \cite{Serfaty2021} for the ``renormalized commutator'' \eqref{eq:rcomm}, which we have reproduced below in the form of \cref{prop:com}--specializing to the notation of our setting.

\begin{prop}
\label{prop:com}
Let $d\geq 2$. There exists a constant $C_d>0$ such that for any Lipschitz vector field $v:\T^d\rightarrow\R^d$, $\mu \in L^\infty(\T^d)$, and vector $\ux_N\in (\T^d)^N \setminus\D_N$, we have the estimate
\begin{equation}
\label{eq:com}
\begin{split}
&\left|\int_{(\T^d)^2\setminus\D_2} (v(x)-v(y))\cdot\nabla\g(x-y)d(\frac{1}{N}\sum_{i=1}^N\d_{x_i}-\mu)^{\otimes 2}(x,y)\right|\\
&\leq C_d \|\nabla v\|_{L^\infty}\Bigg(|\Fr_N(\ux_N,\mu)|+\frac{|\ln \ep|}{N}1_{d=2} + \frac{1}{N\ep^{d-2}}1_{d\geq 3} + (1+\|\mu\|_{L^\infty})\ep^2\Bigg),
\end{split}
\end{equation}
for all choices $0<\ep<1/8$.
\end{prop}

Optimizing the choice of $\ep$ by taking $\ep=N^{-1/d}$, this result gives a sharper $N$-dependence for the additive error terms in the right-hand side of \eqref{eq:com} compared to the aforementioned works \cite{Serfaty2020, Rosenzweig2020_SPVMF}. The proof is much more involved, though, and relies on a previously unsused cancellation in the difference $\frac{1}{N}\sum_{i=1}^N\d_{x_i}-\mu$. The exact statement of \cref{prop:com} is not contained in \cite{Serfaty2021}, but it follows from Proposition 4.3 and Remark 1.2 in that work. We give below a new, shorter proof that bypasses the variation of energy argument used in that work. To this end, we first record a lemma that will be an important ingredient in the proof.

\begin{lemma}
\label{lem:HLinf}
For $d\geq 2$, let $\ux_N\in (\T^d)^N\setminus\D_N$ and $\mu\in L^\infty(\T^d)$. For $0< \eta_i\leq r_{i,\ep}$, define the function
\begin{equation}
H_{N,\ueta_N,i}^{\mu,\ux_N} \coloneqq \g\ast(\frac{1}{N}\sum_{{1\leq j\leq N}\atop {j\neq i}}\d_{x_j}^{(\eta_j)} - \mu),
\end{equation}
where we use the convention that $\d_{x_j}^{(0)} = \d_{x_j}$. Then
\begin{equation}
\|\nabla H_{N,\ueta_N,i}^{\mu,\ux_N}\|_{L^\infty(B(x_i,\eta_i))} \lesssim_{d} \frac{1}{\eta_i^{d/2}}\paren*{\int_{B(x_i,2\eta_i)}  |\nabla H_{N,\ueta_N,i}^{\mu,\ux_N}(x)|^2 dx}^{1/2} + \|\mu\|_{L^\infty}\eta_i.
\end{equation}
\end{lemma}
\begin{proof}
See \cite[Lemma A.2]{Serfaty2021}.
\end{proof}

\begin{remark}
\label{rem:Hi}
Note that since the balls $B(x_{i},r_{i,\ep})$, $1\leq i\leq N$, are pairwise disjoint and $\g_\eta(x)$ is constant for $|x|\leq\eta$ and equals $\g(x)$ otherwise, we have for any choices $0<\eta_i\leq r_{i,\ep}$,
\begin{equation}
\|\nabla H_{N,\ueta_N}^{\mu,\ux_N}\|_{L^2}^2 = \int_{\T^d\setminus\bigcup_{i=1}^N B(x_i,\eta_i)}|\nabla H_{N}^{\mu,\ux_N}(x)|^2dx + \sum_{i=1}^N \int_{B(x_i,\eta_i)} |\nabla H_{N,\ueta_N,i}^{\mu,\ux_N}(x)|^2dx.
\end{equation}
Also note from the convexity of the function $z\mapsto |z|^2$ that
\begin{equation}
|\nabla H_{N,\ueta_N,i}^{\mu,\ux_N}(x)|^2 \lesssim |\nabla H_{N,\ueta_N}^{\mu,\ux_N}(x)|^2 + \frac{1}{N^2}|\nabla\g_{\eta_i}(x-x_i)|^2,
\end{equation}
and therefore,
\begin{align}
\sum_{i=1}^N \int_{B(x_i,2\eta_i)} |\nabla H_{N,\ueta_N,i}^{\mu,\ux_N}(x)|^2dx &\lesssim \|\nabla H_{N,\ueta_N}^{\mu,\ux_N}\|_{L^2}^2 + \frac{1}{N^2}\sum_{i=1}^N \int_{B(x_i,2\eta_i)}	|\nabla\g_{\eta_i}(x-x_i)|^2dx \nn\\
&\lesssim_d \|\nabla H_{N,\ueta_N}^{\mu,\ux_N}\|_{L^2}^2 + \frac{1}{N^2}\sum_{i=1}^N\eta_i^{2-d}.
\end{align}
\end{remark}

\begin{proof}[Proof of \cref{prop:com}]
For shorthand, let us introduce the kernel notation
\begin{equation}
\label{eq:Kvdef}
K_v(x,y) \coloneqq (v(x)-v(y))\cdot\nabla\g(x-y), \qquad x\neq y\in \T^d.
\end{equation}
Introducing a parameter vector $\ueta_N\in (\R_+)^N$ to be specified momentarily, we observe the identity
\begin{equation}
\int_{(\T^d)^2\setminus\D_2} K_v(x,y)d(\frac{1}{N}\sum_{i=1}^N\d_{x_i}-\mu)^{\otimes 2}(x,y) = \sum_{i=1}^3 \Te_i,
\end{equation}
where
\begin{align}
\Te_1 &\coloneqq \int_{(\T^d)^2\setminus\D_2} K_v(x,y)d(\frac{1}{N}\sum_{i=1}^N\d_{x_i}^{(\eta_i)}-\mu)^{\otimes 2}(x,y),\\
\Te_2 &\coloneqq \frac{1}{N}\sum_{j=1}^N\int_{(\T^d)^2\setminus\D_2} K_v(x,y)d(\frac{1}{N}\sum_{i=1}^N\d_{x_i}^{(\eta_i)}- \mu)(x)d( \d_{x_j}-\d_{x_j}^{(\eta_j)})(y),\\
\Te_3 &\coloneqq \frac{1}{N}\sum_{j=1}^N\int_{(\T^d)^2\setminus\D_2} K_v(x,y)d(\frac{1}{N}\sum_{i=1}^N\d_{x_i}- \mu)(x)d( \d_{x_j}-\d_{x_j}^{(\eta_j)})(y).\\
\end{align}
We now estimate each of the $\Te_i$.

\begin{description}[leftmargin=*]
\item[$\Te_1$] 
Using \cite[Lemma 4.3]{Serfaty2020}, we can rewrite $\Te_1$ as
\begin{equation}
\Te_1 = \int_{\T^d} \nabla v(x): \comm{H_{N,\ueta_N}^{\mu,\ux_N}}{H_{N,\ueta_N}^{\mu,\ux_N}}_{SE}(x)dx,
\end{equation}
where $\comm{H_{N,\ueta_N}^{\mu,\ux_N}}{H_{N,\ueta_N}^{\mu,\ux_N}}_{SE}$ is the ($d\times d$) stress-energy tensor defined by
\begin{equation}
\label{eq:SE}
\comm{H_{N,\ueta_N}^{\mu,\ux_N}}{H_{N,\ueta_N}^{\mu,\ux_N}}_{SE}^{\al\be} \coloneqq 2\p_\al H_{N,\ueta_N}^{\mu,\ux_N} \p_\be H_{N,\ueta_N}^{\mu,\ux_N} -|\nabla H_{N,\ueta_N}^{\mu,\ux_N}|^2\d_{\al\be}, \qquad \al,\be\in\{1,\ldots,d\}.
\end{equation}
So by Cauchy-Schwarz and the triangle inequality, it follows that
\begin{equation}
\label{eq:commT1fin}
|\Te_1| \lesssim_d \|\nabla v\|_{L^\infty} \|\nabla H_{N,\ueta_N}^{\mu,\ux_N}\|_{L^2}^2.
\end{equation}

\item[$\Te_2$] 
Recalling that $\d_{x_j}^{(\eta_j)}$ is a probability measure, we rewrite $\Te_2$ as
\begin{equation}
\begin{split}
&\frac{1}{N}\sum_{j=1}^N\int_{(\T^d)^2}\paren*{K_v(x,x_j)-K_v(x,y)} d(\frac{1}{N}\sum_{{1\leq i\leq N}\atop{i\neq j}}\d_{x_i}^{(\eta_i)}-\mu)(x)d\d_{x_j}^{(\eta_j)}(y) \\
&\ph - \frac{1}{N^2}\sum_{j=1}^N \int_{\T^d} \paren*{K_v(x,x_j)-K_v(x,y)}d(\d_{x_j}^{(\eta_j)})^{\otimes 2}(x,y).
\end{split}
\end{equation}
Unpacking the definition \eqref{eq:Kvdef} of $K_v$, we add and subtract to write
\begin{equation}
K_v(x,x_j) - K_v(x,y) = \paren*{\nabla\g(x-x_j) -\nabla\g(x-y)}\cdot \paren*{v(x)-v(x_j)} + \nabla\g(x-y)\cdot \paren*{v(y)-v(x_j)}.
\end{equation}
Recalling identity \eqref{eq:idgconv}, Fubini-Tonelli implies that
\begin{equation}
\begin{split}
&\int_{(\T^d)^2}\paren*{\nabla\g(x-x_j) -\nabla\g(x-y)}\cdot \paren*{v(x)-v(x_j)}d(\frac{1}{N}\sum_{{1\leq i\leq N}\atop{i\neq j}}\d_{x_i}^{(\eta_i)}-\mu)(x)d\d_{x_j}^{(\eta_j)}(y)\\
&=\int_{\T^d} \nabla \f_{\eta_j}(x-x_j)\cdot \paren*{v(x)-v(x_j)}d(\frac{1}{N}\sum_{{1\leq i\leq N}\atop{i\neq j}}\d_{x_i}^{(\eta_i)}-\mu)(x).
\end{split}
\end{equation}
Assuming $\eta_j\leq r_{j,\ep}$, $\supp(\nabla \f_{\eta_j})(\cdot-x_j)\subset \ol{B(x_j,\eta_j)}$ implies that the second line simplifies to
\begin{equation}
-\int_{\T^d} \nabla \f_{\eta_j}(x-x_j)\cdot \paren*{v(x)-v(x_j)}d\mu(x).
\end{equation}
Next, we use Fubini-Tonelli to write
\begin{equation}
\begin{split}
&\int_{(\T^d)^2}\nabla\g(x-y)\cdot \paren*{v(y)-v(x_j)}d(\frac{1}{N}\sum_{{1\leq i\leq N}\atop{i\neq j}}\d_{x_i}^{(\eta_i)}-\mu)(x)d\d_{x_j}^{(\eta_j)}(y) \\
&=-\int_{\T^d}\nabla H_{N,\ueta_N,j}^{\mu,\ux_N}(y)\cdot \paren*{v(y)-v(x_j)}d\d_{x_j}^{(\eta_j)}(y).
\end{split}
\end{equation}
After a little bookkeeping, we find that
\begin{equation}
\label{eq:T2bkp}
\begin{split}
\Te_2 &= -\frac{1}{N}\sum_{j=1}^N\Bigg(\int_{\T^d} \nabla \f_{\eta_j}(x-x_j)\cdot \paren*{v(x)-v(x_j)}d\mu(x)\\
&\ph  + \int_{\T^d}\nabla H_{N,\ueta_N,j}^{\mu,\ux_N}(y)\cdot \paren*{v(y)-v(x_j)}d\d_{x_j}^{(\eta_j)}(y) \\
&\ph\ph+ \frac{1}{N} \int_{\T^d} \paren*{K_v(x,x_j)-K_v(x,y)}d(\d_{x_j}^{(\eta_j)})^{\otimes 2}(x,y)\Bigg).
\end{split}
\end{equation}

We proceed to estimate each of the terms in the right-hand side of \eqref{eq:T2bkp}. Since $\supp(\nabla \f_{\eta_j}(\cdot-x_j))\subset \ol{B(x_j,\eta_j)}$, we can use the mean-value theorem and H\"older's inequality to obtain
\begin{align}
\left|\int_{\T^d} \nabla \f_{\eta_j}(x-x_j)\cdot \paren*{v(x)-v(x_j)}d\mu(x)\right| \leq \|\nabla \f_{\eta_j}\|_{L^1}\|\nabla v\|_{L^\infty} \|\mu\|_{L^\infty}\eta_j &\lesssim_d \|\nabla v\|_{L^\infty}\|\mu\|_{L^\infty}\eta_j^2,
\end{align}
where the ultimate line follows from \cref{lem:gLp}. Since $\supp(\d_{x_j}^{(\eta_j)})\subset \ol{B(x_j,\eta_j)}$, we can argue similarly to before, obtaining
\begin{align}
&\left|\int_{\T^d}\nabla H_{N,\ueta_N,j}^{\mu,\ux_N}(y)\cdot \paren*{v(y)-v(x_j)}d\d_{x_j}^{(\eta_j)}(y) \right|\nn\\
&\leq \|\nabla H_{N,\ueta_N,j}^{\mu,\ux_N}\|_{L^\infty(B(x_j,\eta_j))}\|\nabla v\|_{L^\infty}\eta_j \nn\\
&\lesssim_d \|\nabla v\|_{L^\infty}\eta_j\paren*{\frac{1}{\eta_j^{d/2}}\paren*{\int_{B(x_j,2\eta_j)}  |\nabla H_{N,\ueta_N,j}^{\mu,\ux_N}(x)|^2 dx}^{1/2} +\|\mu\|_{L^\infty}\eta_{j}},
\end{align}
where the ultimate line follows from \cref{lem:HLinf}. Finally, using the crude bound
\begin{equation}
|K_v(x,y)| \lesssim_d \min\{\frac{\|\nabla v\|_{L^\infty}}{|x-y|^{d-2}}, \frac{\|v\|_{L^\infty}}{|x-y|^{d-1}}\},
\end{equation}
which follows from the definition \eqref{eq:Kvdef} of $K_v(x,y)$ and the mean-value theorem, together with the scaling invariance of Lebesgue measure, we find that
\begin{equation}
\left|\int_{\T^d} \paren*{K_v(x,x_j)-K_v(x,y)}d(\d_{x_j}^{(\eta_j)})^{\otimes 2}(x,y)\right|\lesssim_d \|\nabla v\|_{L^\infty}\eta_j^{2-d}.
\end{equation}
Putting together the above estimates, we have shown that
\begin{equation}
\label{eq:commT2fin}
\begin{split}
|\Te_2| &\lesssim_d \frac{\|\nabla v\|_{L^\infty}}{N}\sum_{j=1}^N\Bigg(\|\mu\|_{L^\infty}\eta_j^2 + \frac{1}{N\eta_{j}^{d-2}}+ \eta_j^{1-\frac{d}{2}}\paren*{\int_{B(x_j,2\eta_j)}  |\nabla H_{N,\ueta_N,j}^{\mu,\ux_N}(x)|^2 dx}^{1/2}\Bigg).
\end{split}
\end{equation}
\item[$\Te_3$] 
The analysis is essentially the same as for that of $\Te_2$, therefore we omit the details. Ultimately, we find that $\Te_3$ is controlled by the right-hand side of \eqref{eq:commT2fin}.
\end{description}

Putting together the estimates \eqref{eq:commT1fin} and \eqref{eq:commT2fin}, we have shown that the left-hand side of \eqref{eq:com} is $\lesssim_d$
\begin{equation}
\begin{split}
&\|\nabla v\|_{L^\infty} \|\nabla H_{N,\ueta_N}^{\mu,\ux_N}\|_{L^2}^2 + \frac{\|\nabla v\|_{L^\infty}}{N}\sum_{j=1}^N\Bigg(\|\mu\|_{L^\infty}\eta_j^2 + \frac{1}{N\eta_{j}^{d-2}} \\
&\ph + \eta_j^{1-\frac{d}{2}}\paren*{\int_{B(x_j,2\eta_j)}  |\nabla H_{N,\ueta_N,j}^{\mu,\ux_N}(x)|^2 dx}^{1/2} \Bigg).
\end{split}
\end{equation}
By Cauchy-Schwarz in the $j$-summation,
\begin{align}
&\frac{1}{N}\sum_{j=1}^N \eta_j^{1-\frac{d}{2}}\paren*{\int_{B(x_j,2\eta_j)}  |\nabla H_{N,\ueta_N,j}^{\mu,\ux_N}(x)|^2 dx}^{1/2} \nn\\
&\leq \paren*{\frac{1}{N^2}\sum_{j=1}^N \eta_j^{2-d}}^{1/2} \paren*{\sum_{j=1}^N \int_{B(x_j,2\eta_j)}  |\nabla H_{N,\ueta_N,j}^{\mu,\ux_N}(x)|^2 dx}^{1/2} \nn\\
&\lesssim_d \paren*{\frac{1}{N^2}\sum_{j=1}^N \eta_j^{2-d}}^{1/2}\|\nabla H_{N,\ueta_N}^{\mu,\ux_N}\|_{L^2} + \paren*{\frac{1}{N^2}\sum_{j=1}^N \eta_j^{2-d}},
\end{align}
where we use \cref{rem:Hi} to obtain the ultimate line. We choose $\eta_i=r_{i,\ep}$ (recall definition \eqref{eq:ridef}) for every $1\leq i\leq N$, so that by estimate \eqref{eq:MEsi} of \cref{cor:MEct}, we have that
\begin{equation}
N^{-2}\sum_{j=1}^N  r_{j,\ep}^{2-d} \lesssim_d |\Fr_N(\ux_N,\mu)| + \frac{1}{N}(|\ln \ep|1_{d=2} + \ep^{2-d}1_{d\geq 3}) + (1+\|\mu\|_{L^\infty})\ep^2.
\end{equation}
Similarly, by estimate \eqref{eq:MEHlb} of \cref{cor:MEct}, we also have the bound
\begin{equation}
\|\nabla H_{N,\ur_{N,\ep}}^{\mu,\ux_N}\|_{L^2}^2 \lesssim_d |\Fr_N(\ux_N,\mu)| + \frac{1}{N}(|\ln \ep|1_{d=2} + \ep^{2-d}1_{d\geq 3}) + (1+\|\mu\|_{L^\infty})\ep^2.
\end{equation}
Using that $r_{i,\ep}\leq \ep$ tautologically and performing a little algebra, we arrive at the desired conclusion.
\end{proof}

\section{Proof of main results}
\label{sec:MR}
In this section, we give the proof of our main results, \cref{thm:main} and \cref{cor:main}. We have divided the proof into two subsections. In \cref{ssec:MRme}, we establish the Gronwall-type estimate for the modulated energy $\Hr_{N,\vep}(\uz_N(t),u(t))$, which proves \cref{thm:main}. In \cref{ssec:MR_W} we show how to deduce the weak-* convergence of the empirical measure from the modulated-energy estimate, which proves \cref{cor:main}.

\subsection{Modulated energy estimate}
\label{ssec:MRme}
We recall from \cref{ssec:introPf} of the introduction that the time derivative of the modulated energy satisfies the identify
\begin{equation}
\frac{d}{dt}\Hr_{N,\vep}(\uz_N(t),u(t)) = \Te_1 + \cdots + \Te_4,
\end{equation}
where the definitions of $\Te_1,\ldots,\Te_4$ are given in \eqref{eq:MET1def} - \eqref{eq:MET4def}. Through a series of four lemmas, we estimate each of the $\mathrm{Term}_j$, beginning with $\mathrm{Term}_1$. Since all of the estimates are static (i.e. they hold pointwise in time), we shall omit the time-dependence until the end of this subsection. Also, we shall omit the underlying spatial domain $\T^d$ in our function space notation.

\begin{lemma}
\label{lemT1}
For $d\geq 2$,
\begin{equation}
|\mathrm{Term}_1| \lesssim_d \frac{\|\nabla u\|_{L^\infty}}{N}\sum_{i=1}^N |u(x_i)-v_i|^2.
\end{equation}
\end{lemma}
\begin{proof}
Immediate from taking absolute values.
\end{proof}

\begin{lemma}
\label{lemT2}
For $d\geq 2$,
\begin{equation}
\begin{split}
|\mathrm{Term}_2| &\lesssim_d \frac{\|\nabla u\|_{L^\infty}}{\vep^2} \Bigg(|\Fr_N(\ux_N,1+\vep^2\Uu)|+\frac{1}{N}\paren*{|\ln\ep|1_{d=2} + \ep^{2-d}1_{d\geq 3}} \\
&\ph + (1+\vep^2\|\nabla u\|_{L^\infty}^2)\ep^2 \Bigg),
\end{split}
\end{equation}
\end{lemma}
for all choices $0<\ep<1/8$
\begin{proof}
Apply \cref{prop:com} with $v=u$ and $\mu = 1+\vep^2\Uu$ and use that $\|\Uu\|_{L^\infty} \lesssim_d \|\nabla u\|_{L^\infty}^2$.
\end{proof}

\begin{lemma}
\label{lemT3}
For $d\geq 2$ and $s>0$, there exists a constant $C_{d,s}>0$ such that 
\begin{equation}
\begin{split}
|\Te_3| &\leq C_{d,s}\|u\|_{B_{\infty,\infty}^{1+s}}^3\Bigg(\Fr_N(\ux_N,1+\vep^2 \Uu) + \frac{C_{d,s}}{N}\Big(|\ln\ep|1_{d=2} + \ep^{2-d}1_{d\geq 3}\Big)\\
&\ph + C_{d,s}(1+\vep^2\|\nabla u\|_{L^\infty}^2)\ep^2 \Bigg)^{1/2}.
\end{split}
\end{equation}
for all choices $0<\ep<1/8$.
\end{lemma}
\begin{proof}
Applying \cref{prop:MESob} with test function $\varphi=\nabla\Dm^{-2}(u\Uu)$ and measure $\mu = 1+\vep^2\Uu$, we find that
\begin{equation}
\begin{split}
|\Te_3| &\lesssim_d \ep\|\nabla^{\otimes 2}\Dm^{-2}(u\Uu)\|_{L^\infty}  + \|\nabla^{\otimes 2}\Dm^{-2}(u\Uu)\|_{L^2}\Bigg(\Fr_N(\ux_N,1+\vep^2 \Uu) \\
&\ph\ph + \frac{C_d}{N}\Big(|\ln\ep|1_{d=2} + \ep^{2-d}1_{d\geq 3}\Big) + C_d(1+\vep^2\|\Uu\|_{L^\infty})\ep^2 \Bigg)^{1/2}
\end{split}
\end{equation}
From the definition of the $B_{\infty,1}^0$ norm and the fact that $\nabla^{\otimes 2}\Dm^{-2}(u\U)$ has Fourier support away from the origin, it follows  that
\begin{equation}
\|\nabla^{\otimes 2}\Dm^{-2}(u\Uu)\|_{L^\infty} \lesssim_d \|u\Uu\|_{B_{\infty,1}^0} \lesssim_{s,d} \|u\Uu\|_{B_{\infty,\infty}^s}
\end{equation}
for any $s>0$. Using \cref{prop:Bony}, it follows that
\begin{equation}
\|u\Uu\|_{B_{\infty,\infty}^s} \lesssim_{s,d}\|u\|_{B_{\infty,\infty}^{1+s}}^3.
\end{equation}
Since $\|\cdot\|_{L^2}\leq\|\cdot\|_{L^\infty}$ by H\"older's inequality, we see that the proof is complete.
\end{proof}

\begin{lemma}
\label{lemT4}
For $d\geq 2$ and $s>0$, there exists a constant $C_{d,s}>0$ such that
\begin{equation}
\begin{split}
|\Te_4| &\leq C_{d,s}\|u\|_{B_{\infty,\infty}^{1+s}}^3\Bigg(\Fr_N(\ux_N,1+\vep^2 \Uu) + \frac{C_{d,s}}{N}\Big(|\ln\ep|1_{d=2} + \ep^{2-d}1_{d\geq 3}\Big)\\
&\ph + C_{d,s}(1+\vep^2 \|\nabla u\|_{L^\infty}^2)\ep^2 \Bigg)^{1/2}.
\end{split}
\end{equation}
for all choices $0<\ep<1/8$.
\end{lemma}
\begin{proof}
The proof is similar to that of \cref{lemT3}. Applying \cref{prop:MESob} with test function $\varphi=\p_t p$ and measure $\mu = 1+\vep^2\Uu$, we find that
\begin{equation}
\begin{split}
|\Te_4| &\lesssim_d \ep\|\nabla \p_t p\|_{L^\infty} + \|\nabla\p_tp\|_{L^2}\Bigg(\Fr_N(\ux_N,1+\vep^2 \Uu) + \frac{C_d}{N}\Big(|\ln\ep|1_{d=2} + \ep^{2-d}1_{d\geq 3}\Big) \\
&\ph + C_d(1+\vep^2\|\Uu\|_{L^\infty})\ep^2 \Bigg)^{1/2}.
\end{split}
\end{equation}
By direct computation, one finds that the pressure $p$ satisfies the equation
\begin{equation}
\p_t p = 2\p_\al\Dm^{-2}(u^\gamma\p_\gamma u^\beta\p_\beta u^\al) + 2\Dm^{-2}\paren*{\p_\al\p_\be\Dm^{-2}\Uu\p_\be u^\al},
\end{equation}
where we use the convention of Einstein summation. We now proceed similarly as to in the proof of the previous lemma to find that
\begin{align}
\|\p_t\nabla p\|_{L^\infty} \lesssim_{s,d} \|u\|_{B_{\infty,\infty}^{1+s}}^3.
\end{align}
Since $\|\cdot\|_{L^2}\leq\|\cdot\|_{L^\infty}$ by H\"older, the proof is complete.
\end{proof}

Using the triangle inequality and applying \Cref{lemT1,lemT2,lemT3,lemT4}, we find that there are constants $C_d, C_{d,s}>0$ such that
\begin{equation}
\begin{split}
&|\Hr_{N,\vep}(\uz_N(t),u(t))| \\
&\leq |\Hr_{N,\vep}(\uz_N(0),u(0))| +  \int_0^t \frac{C_d\|\nabla u(\tau)\|_{L^\infty}}{N}\sum_{i=1}^N |u(\tau,x_i(\tau))-v_i(\tau)|^2 d\tau\\
&\ph + \int_0^t \frac{C_d\|\nabla u(\tau)\|_{L^\infty}}{\vep^2} \Bigg(|\Fr_N(\ux_N(\tau),1+\vep^2\Uu(\tau))|+\frac{1}{N}\paren*{|\ln\ep|1_{d=2} + \ep^{2-d}1_{d\geq 3}} \\
&\ph\ph + (1+\vep^2\|\nabla u(\tau)\|_{L^\infty}^2)\ep^2 \Bigg)d\tau \\
&\ph +  \int_0^t C_{d,s}\|u(\tau)\|_{B_{\infty,\infty}^{1+s}}^3\Bigg(\Fr_N(\ux_N(\tau),1+\vep^2 \Uu(\tau)) + \frac{C_{d,s}}{N}\Big(|\ln\ep|1_{d=2} + \ep^{2-d}1_{d\geq 3}\Big)\\
&\ph + C_{d,s}(1+\vep^2 \|\nabla u(\tau)\|_{L^\infty}^2)\ep^2 \Bigg)^{1/2}d\tau.
\end{split}
\end{equation}
Remembering the definition \eqref{eq:MEdef} of $\Hr_{N,\vep}(\uz_N,u)$ and simplifying a little, the preceding inequality implies
\begin{equation}
\label{eq:rhsub}
\begin{split}
&|\Hr_{N,\vep}(\uz_N(t),u(t))| \\
&\leq |\Hr_{N,\vep}(\uz_N(0),u(0))| + C_{d,s}\int_0^t (1+\|\nabla u(\tau)\|_{L^\infty}) |\Hr_{N,\vep}(\uz_N(\tau),u(\tau))| d\tau + C_{d,s}\vep^2\int_0^t \|u(\tau)\|_{B_{\infty,\infty}^{1+s}}^6 d\tau\\
&\ph + \frac{C_{d,s}}{\vep^2}\int_0^t (1+\|\nabla u(\tau)\|_{L^\infty})\Bigg(\frac{1}{N}(|\ln\ep|1_{d=2} + \ep^{2-d}1_{d\geq 3}) + (1+\vep^2\|\nabla u(\tau)\|_{L^\infty}^2)\ep^2 \Bigg)d\tau,
\end{split}
\end{equation}
where $C_{d,s}>0$ is a possibly larger constant.

To complete the proof of \cref{thm:main}, we now balance the terms in the ultimate line by choosing $\ep=N^{-1/d}/8$. This choice satisfies the constraint $0<\ep<1/8$ and is such that all terms in the ultimate line are of the same order in $N$ (up to a logarithmic factor if $d=2$). Substituting this choice of $\ep$ into the right-hand side of inequality \eqref{eq:rhsub} and applying the Gronwall-Bellman lemma, we see that the proof of the theorem is complete.

\subsection{Convergence of the empirical measure}
\label{ssec:MR_W}
Finally, we show that the quantitative estimate \eqref{eq:main} for the modulated energy $\Hr_{N,\vep}(\uz_N,u)$ implies weak-* convergence of the empirical measure to the limiting measure \eqref{eq:limEM}. This then proves \cref{cor:main}.

It suffices to show that if $\varphi\in C(\T^d\times\R^d)$ is such that for every $\kappa>0$, there exists a compact set $K_\kappa\subset \R^d$ so that
\begin{equation}
\sup_{x \in \T^d, v\in K_\kappa^c} |\varphi(x,v)| \leq \kappa,
\end{equation}
then
\begin{equation}
\label{eq:vphirep}
\sup_{0\leq t\leq T} \left|\frac{1}{N}\sum_{i=1}^N \varphi(x_i(t),v_i(t)) - \int_{\T^d}\varphi(x,u(t,x))dx \right| \xrightarrow[N\rightarrow\infty]{} 0.
\end{equation}
Let $\chi$ be a compactly supported mollifier on $\R^d$, and let 
\begin{equation}
\varphi_\eta(x,v) \coloneqq \eta^{-2d}\int_{\T^d\times\R^d}\varphi(x-x',v-v')\chi(\frac{x'}{\eta})\chi(\frac{v'}{\eta})dxdv.
\end{equation}
Since for any compact set $K'\subset \R^d$, $\T^d\times K'$ is compact, we have that
\begin{equation}
\sup_{(x,v)\in \T^d\times K'} |\varphi(x,v)-\varphi_\eta(x,v)| \xrightarrow[\eta\rightarrow 0^+]{} 0.
\end{equation}
It follows now from the triangle inequality that we can replace $\varphi$ in \eqref{eq:vphirep} with $\varphi_\eta$ at the expense of picking up a term which can be made arbitrarily small by taking $\eta$ arbitrarily small. So without loss of generality, we may assume that $\varphi\in C^\infty(\T^d\times\R^d)$.

With this assumption, we use the triangle inequality to control the left-hand side of \eqref{eq:vphirep} by
\begin{equation}
\begin{split}
&\sup_{0\leq t\leq T}\Bigg( \frac{1}{N}\sum_{i=1}^N |\varphi(x_i(t),v_i(t)) - \varphi(x_i(t),u(t,x_i(t)))| \\
&\ph+\left|\int_{\T^d}\varphi(x,u(t,x))d(\frac{1}{N}\sum_{i=1}^N \d_{x_i(t)}-1-\vep^2\Uu(t))(x)\right|  + \vep^2\|\varphi\|_{L_{x,v}^\infty}\|\nabla u(t)\|_{L^\infty}^2\Bigg).
\end{split}
\end{equation}
By the mean-value theorem in the velocity variable,
\begin{align}
\frac{1}{N}\sum_{i=1}^N |\varphi(x_i(t),v_i(t)) - \varphi(x_i(t),u(t,x_i(t)))| &\leq \frac{\|\nabla_v\varphi\|_{L^\infty}}{N}\sum_{i=1}^N |v_i(t)-u(t,x_i(t))| \nn\\
&\leq \frac{\|\nabla_v\varphi\|_{L^\infty}}{N^{1/2}}\paren*{\sum_{i=1}^N |v_i(t)-u(t,x_i(t))|^2}^{1/2}.
\end{align}
By \cref{prop:MESob} applied with $\mu = 1+\vep^2\Uu$ and test function $\tl{\varphi}(x) \coloneqq \varphi(x,u(x))$,
\begin{equation}
\begin{split}
&\left|\int_{\T^d}\varphi(x,u(t,x))d(\frac{1}{N}\sum_{i=1}^N \d_{x_i(t)}-1-\vep^2\Uu(t))(x)\right| \\
&\lesssim \|\nabla\tl{\varphi}\|_{L^\infty}\Bigg(\Fr_N(\ux_N,1+\vep^2 \Uu) + \frac{C_d}{N}\Big(|\ln\ep|1_{d=2} + \ep^{2-d}1_{d\geq 3}\Big)  + C_d(1+\vep^2 \|\nabla u\|_{L^\infty}^2)\ep^2 \Bigg)^{1/2}
\end{split}
\end{equation}
for some constant $C_d>0$ and all choices $0<\ep<1/8$. We optimize by choosing $\ep= N^{-1/d}$. Note that by the chain rule,
\begin{equation}
\|\nabla\tl{\varphi}\|_{L^\infty} \leq \|\nabla_x\varphi\|_{L_{x,v}^\infty} + \|\nabla_v\varphi\|_{L_{x,v}^\infty}\|\nabla u\|_{L^\infty}.
\end{equation}
Letting $N\rightarrow\infty$ and appealing to \cref{thm:main}, a little bookkeeping reveals that all terms above vanish uniformly on the interval $[0,T]$. Therefore, the proof is complete.

\bibliographystyle{siam}
\bibliography{NewEul}

\end{document}